\documentclass[11pt,a4paper]{article}
\usepackage[latin1]{inputenc}
\usepackage{amsmath}
\usepackage{amsthm}
\usepackage{amsfonts}
\usepackage{amsfonts,amsthm,latexsym,amsmath,amssymb,amscd,epsfig,psfrag,enumerate}
\usepackage{graphics,graphicx, bezier, float, color, hyperref}
\usepackage{amssymb,url}
\usepackage{multienum}
\usepackage[table]{xcolor}
\usepackage{multicol,multirow}
\usepackage{graphicx}
\usepackage{fancyvrb}
\usepackage{parskip}
\usepackage{listings}
\usepackage[toc,page]{appendix}
\usepackage[top=2.7cm, bottom=2.7cm, left=1.5cm, right=1.5cm]{geometry}
\usepackage{xcolor}

\usepackage{blkarray}
\newtheorem{theorem}{Theorem}[section]
\newtheorem{lemma}{Lemma}[section]

\usepackage[none]{hyphenat}[section]
\newtheorem{definition}{Definition}[section]

\numberwithin{equation}{section}
\numberwithin{table}{section}
\numberwithin{figure}{section}
                              
\title{On Palindromic forms in the $k$-Lucas sequence composed of two distinct Repdigits}
\author{Herbert Batte$^{1} $ and Prosper Kaggwa$^{2,*}$}
\date{}                       
                              
\begin{document}              
\maketitle                   
\abstract{ 
For integers $k \geq 2$, the $k$-generalized Lucas sequence $\{L_n^{(k)}\}_{n \geq 2-k}$ is defined by the recurrence relation  
\[                            
L_n^{(k)} = L_{n-1}^{(k)} + \cdots + L_{n-k}^{(k)} \quad \text{for } n \geq 2, 
\]                            
with initial terms given by $L_0^{(k)} = 2$, $L_1^{(k)} = 1$, and $L_{2-k}^{(k)} = \cdots = L_{-1}^{(k)} = 0$.  
In this paper, we extend work in \cite{Lucas} and show that the result in \cite{Lucas} still holds for $k\ge 3$, that is, we show that for $k\ge 3$, there is no $k$-generalized Lucas number appearing as a palindrome formed by concatenating two distinct repdigits.
}            
	                             
	                             
	                             
	\thanks{$ ^{*} $ Corresponding author}
	                             
\section{Introduction}        
\subsection{Background}       
                              
For integers \( k \geq 2 \), the \( k \)-generalized Lucas sequence \( \{L_n^{(k)}\}_{n \in \mathbb{Z}} \), is defined by the recurrence relation \[ L_n^{(k)} = L_{n-1}^{(k)} + \cdots + L_{n-k}^{(k)} \] for \( k \geq 2 \), with initial conditions \( L_0^{(k)} = 2 \), \( L_1^{(k)} = 1 \), and \( L_{2-k}^{(k)} = \cdots = L_{-1}^{(k)} = 0 \). It is a generalization of the well-known sequence of Lucas numbers, which is obtained in the case \( k = 2 \) and has been the subject of extensive research in Number Theory.
                              
A \textit{repdigit} in base 10 refers to a number consisting solely of repeated instances of the same digit. It can be expressed as
\[                            
N = \underbrace{\overline{d \ldots d}}_{\ell \text{ times}} = d \left( \frac{10^\ell - 1}{9} \right),
\]
where \( 0 \leq d \leq 9 \) and \( \ell \geq 1 \). 

In recent years, there has been growing interest in understanding the Diophantine properties of recurrence sequences, particularly in relation to their connection with repdigits. Several studies have explored whether certain elements of these sequences can be expressed as either sums or concatenations involving specific types of numbers. Early contributions by Banks and Luca in \cite{banks} offered foundational results on this topic, albeit within a limited framework. A more targeted investigation by \cite{Repdigit2} focused on Lucas numbers which are concatenations of
two repdigits, establishing that the only such instances are 
$$L_6 =18,\quad L_7 = 29, \quad L_{8} =47,\quad L_9=96,\quad L_{11}=199 \quad \text{and} \quad L_{12} = 322. $$

Recently, researched have turned attention on exploring palindromic numbers constructed from repdigits within the framework of linear recurrence sequences. A \textit{palindrome} is an integer that remains the same when read in reverse, for example, 12345678987654321, 909, and 6112116. Interest in identifying palindromes within such sequences has grown considerably, see for example \cite{Padovan}, \cite{Narayan}, \cite{Batte}, \cite{kaggwa}, \cite{guma} and \cite{Lucas}. In particular, the author in \cite{Lucas} proved that there is no Lucas number that can be written as a palindromic concatenation of two distinct repdigits, solving the Diophantine equation \eqref{eq:main} when $k=2$.

Motivated by the increasing interest in this area, we now focus on the \( k \)-Lucas sequence. Specifically, we examine solutions to the Diophantine equation
\begin{align}\label{eq:main}
	L_n^{(k)} = \overline{\underbrace{d_1 \ldots d_1}_{\ell \text{ times}} \underbrace{d_2 \ldots d_2}_{m \text{ times}} \underbrace{d_1 \ldots d_1}_{\ell \text{ times}}},\qquad\text{with}\qquad k\geq3,
\end{align}
where \( d_1, d_2 \in \{0,1,\dots,9\} \), \( d_1 > 0 \), \( d_1 \ne d_2 \), and \( \ell, m \geq 1 \). In other words, we seek those \( k \)-Lucas numbers that take the form of palindromes created by symmetrically concatenating two distinct repdigits. As this structure requires at least three digits, we restrict our attention to \( n \geq 8 \) throughout the paper. Our main result is the following.

\subsection{Main Results}\label{sec:1.2l}
\begin{theorem}\label{thm1.1l} 
	For $k\geq3$, There is no  $k$-Lucas number that is a palindromic concatenation of two
	distinct repdigits.
	\end{theorem}
\section{Some properties of $k$-generalized Lucas numbers.}
It is well established that  
\begin{align}\label{eq2.2}
	L_n^{(k)} = 3 \cdot 2^{n-2},\qquad \text{for all}\qquad 2 \leq n \leq k.
\end{align}
They generate a linearly recurrent sequence with the characteristic polynomial  
\[
\Psi_k(x) = x^k - x^{k-1} - \cdots - x - 1,
\]
which is irreducible over $\mathbb{Q}[x]$. The polynomial $\Psi_k(x)$ has a unique real root $\alpha(k) > 1$, with all other roots lying outside the unit circle, as noted in \cite{miles}. The root $\alpha(k) := \alpha$ satisfies the bound  
\begin{align}\label{3.1}
	2(1 - 2^{-k}) < \alpha < 2 \quad \text{for all} \quad k \geq 2,
\end{align}
as established in \cite{wolfram}. Similar to the classical case where $k=2$, it was demonstrated in \cite{bravo2014} that  
\begin{align}\label{2.2}
	\alpha^{n-1} \leq L_n^{(k)} \leq 2\alpha^n, \quad \text{for all } n \geq 1, k \geq 2.
\end{align}

For any $k \geq 2$, define the function  
\begin{equation}\label{fk}
	f_k(x) := \dfrac{x - 1}{2 + (k+1)(x - 2)}.
\end{equation}
From the bound in \eqref{3.1}, it follows that  
\begin{equation}\label{fkprop}
	\frac{1}{2} = f_k(2) < f_k(\alpha) < f_k(2(1 - 2^k)) \leq \frac{3}{4},
\end{equation}
for all $k \geq 3$. It is straightforward to check that this inequality is also valid for $k = 2$. Moreover, one can verify that for all $2 \leq i \leq k$, where $\alpha_i$ represents the remaining roots of $\Psi_k(x)$, the condition $|f_k(\alpha_i)| < 1$ holds.

The following lemma plays a crucial role in applying Baker's theory.  
\begin{lemma}[Lemma 2, \cite{gomez}]
	For all $k \geq 2$, the number $f_k(\alpha)$ is not an algebraic integer.
\end{lemma}

Additionally, it was proven in \cite{bravo2014} that  
\begin{align}\label{3.5}
	L_n^{(k)} = \sum_{i=1}^k (2\alpha_i - 1) f_k(\alpha_i) \alpha_i^{n-1}, \quad  
	\left|L_n^{(k)} - f_k(\alpha)(2\alpha - 1) \alpha^{n-1} \right| < \frac{3}{2},
\end{align}
for all $k \geq 2$ and $n \geq {2-k}$. This leads to the expression  
\begin{align}\label{3.6}
	L_n^{(k)} = f_k(\alpha)(2\alpha - 1) \alpha^{n-1} + e_k(n), \quad \text{where} \quad |e_k(n)| < 1.5.
\end{align}
The left-hand expression in \eqref{3.5} is commonly referred to as the Binet-like formula for $L_n^{(k)}$. Furthermore, the inequality on the right in \eqref{3.6} indicates that the contribution of roots inside the unit circle to $L_n^{(k)}$ is negligible.
A better estimate than \eqref{3.5} appears in Section $3.3$  of \cite{Hbatte}, but with a more restricted range of $n$ in terms of $k$. It states that 
\begin{align}\label{pk_b1}
	\left| f_k(\alpha)(2\alpha-1)\alpha^{n-1}-3 \cdot2^{n-2}\right| <  3\cdot 2^{n-2}\cdot \frac{36}{2^{k/2}},\qquad {\text{\rm provided}}\qquad n<2^{k/2}.
\end{align}

\section{Methods}
\subsection{Linear Forms in Logarithms}
We utilize four specific cases of Baker-type lower bounds for nonvanishing linear forms in three logarithms of algebraic numbers. Several variants of such bounds are available in the literature, including those established by Baker and W{\"u}stholz \cite{BW} and Matveev \cite{Matveev}. Before presenting these inequalities, we first recall the notion of the height of an algebraic number, defined as follows.

\begin{definition}\label{def2.1l}
	Let \( \gamma \) be an algebraic number of degree \( d \) with minimal primitive polynomial over the integers 
	\[ a_{0}x^{d} + a_{1}x^{d-1} + \cdots + a_{d} = a_{0}\prod_{i=1}^{d}(x - \gamma^{(i)}), \] 
	where the leading coefficient \( a_{0} \) is positive. The logarithmic height of \( \gamma \) is then given by 
	\[ h(\gamma) := \dfrac{1}{d}\left(\log a_{0} + \sum_{i=1}^{d}\log \max\{|\gamma^{(i)}|, 1\} \right). \]
\end{definition}

In the special case where \( \gamma \) is a rational number expressed in lowest terms as \( \gamma = p/q \) with \( q \geq 1 \), we have \( h(\gamma) = \log \max\{|p|, q\} \). The logarithmic height function \( h(\cdot) \) satisfies the following properties, which will be used throughout this paper without further justification:
\begin{equation}\nonumber
	\begin{aligned}
		h(\gamma_{1} \pm \gamma_{2}) &\leq h(\gamma_{1}) + h(\gamma_{2}) + \log 2; \\
		h(\gamma_{1}\gamma_{2}^{\pm 1}) &\leq h(\gamma_{1}) + h(\gamma_{2}); \\
		h(\gamma^{s}) &= |s|h(\gamma) \quad \text{for any integer } s \in \mathbb{Z}.
	\end{aligned}
\end{equation}
Using these properties, it can be computed that 
\begin{align}\label{eqh}
	h\left(f_k(\alpha)\right) < 2\log k, \qquad \text{for all} \quad k\geq2 
\end{align}

A linear form in logarithms is an expression of the form
\begin{equation*}
	\Lambda := b_1\log \gamma_1 + \cdots + b_t\log \gamma_t,
\end{equation*}
where \( \gamma_1, \ldots, \gamma_t \) are positive real algebraic numbers and \( b_1, \ldots, b_t \) are nonzero integers. We assume \( \Lambda \neq 0 \) and seek lower bounds for \( |\Lambda| \). Let \( \mathbb{K} := \mathbb{Q}(\gamma_1, \ldots, \gamma_t) \) and denote by \( D \) the degree of \( \mathbb{K} \). We begin with the general result due to Matveev (Theorem 9.4 in \cite{Matveev}).

\begin{theorem}[Matveev, see Theorem 9.4 in \cite{Matveev}]
	\label{thm:Matl} 
	Let \( \Gamma := \gamma_1^{b_1} \cdots \gamma_t^{b_t} - 1 = e^{\Lambda} - 1 \), and assume \( \Gamma \neq 0 \). Then 
	\[
	\log |\Gamma| > -1.4 \cdot 30^{t+3} \cdot t^{4.5} \cdot D^2 (1 + \log D)(1 + \log B)A_1 \cdots A_t,
	\]
	where \( B \geq \max\{|b_1|, \ldots, |b_t|\} \) and \( A_i \geq \max\{Dh(\gamma_i), |\log \gamma_i|, 0.16\} \) for \( i = 1, \ldots, t \).
\end{theorem}

\subsection{Reduction Techniques}
The bounds obtained from Matveev's theorem are frequently too large for practical computation. To address this, we implement a refinement strategy based on the LLL-algorithm. We first provide some necessary background.

Let \( k \) be a positive integer. A subset \( \mathcal{L} \) of the real vector space \( \mathbb{R}^k \) is called a \emph{lattice} if there exists a basis \( \{b_1, \ldots, b_k\} \) of \( \mathbb{R}^k \) such that
\[
\mathcal{L} = \sum_{i=1}^{k} \mathbb{Z} b_i = \left\{ \sum_{i=1}^{k} r_i b_i \mid r_i \in \mathbb{Z} \right\}.
\]
The vectors \( b_1, \ldots, b_k \) are said to form a basis for \( \mathcal{L} \), and \( k \) is termed the rank of \( \mathcal{L} \). The determinant \( \text{det}(\mathcal{L}) \) is defined as
\[
\text{det}(\mathcal{L}) = | \det(b_1, \ldots, b_k) |,
\]
where the \( b_i \) are represented as column vectors. This value is independent of the choice of basis (see \cite{Cas}, Section 1.2).

For linearly independent vectors \( b_1, \ldots, b_k \) in \( \mathbb{R}^k \), we employ the Gram-Schmidt orthogonalization process to define vectors \( b^*_i \) and coefficients \( \mu_{i,j} \) via
\[
b^*_i = b_i - \sum_{j=1}^{i-1} \mu_{i,j} b^*_j, \quad \mu_{i,j} = \dfrac{\langle b_i, b^*_j \rangle}{\langle b^*_j, b^*_j \rangle},
\]
where \( \langle \cdot, \cdot \rangle \) is the standard inner product on \( \mathbb{R}^k \). Here, \( b^*_i \) represents the orthogonal projection of \( b_i \) onto the complement of the subspace spanned by \( b_1, \ldots, b_{i-1} \), ensuring orthogonality.

\begin{definition}
	A basis \( b_1, \ldots, b_n \) for the lattice \( \mathcal{L} \) is called reduced if
	\begin{align*}
		| \mu_{i,j} | &\leq \frac{1}{2} \quad \text{for} \quad 1 \leq j < i \leq n, \\
		\|b^*_{i} + \mu_{i,i-1} b^*_{i-1}\|^2 &\geq \frac{3}{4} \|b^*_{i-1}\|^2 \quad \text{for} \quad 1 < i \leq n,
	\end{align*}
	where \( \| \cdot \| \) is the Euclidean norm. The constant \( 3/4 \) can be replaced by any fixed \( y \in (1/4, 1) \) (see \cite{LLL}, Section 1).
\end{definition}

For a \( k \)-dimensional lattice \( \mathcal{L} \subseteq \mathbb{R}^k \) with reduced basis \( b_1, \ldots, b_k \), let \( B \) be the matrix with columns \( b_1, \ldots, b_k \). Define
\[
l(\mathcal{L}, y) = \begin{cases}
	\min_{x \in \mathcal{L}} \|x - y\| & \text{if } y \notin \mathcal{L}, \\
	\min_{0 \neq x \in \mathcal{L}} \|x\| & \text{if } y \in \mathcal{L},
\end{cases}
\]
where \( \| \cdot \| \) is the Euclidean norm. The LLL algorithm efficiently computes a lower bound \( l(\mathcal{L}, y) \geq c_1 \) in polynomial time (see \cite{SMA}, Section V.4).

\begin{lemma}\label{lem2.5m}
	Let \( y \in \mathbb{R}^k \) and \( z = B^{-1}y \) with \( z = (z_1, \ldots, z_k)^T \). 
	\begin{enumerate}[(i)]
		\item If \( y \notin \mathcal{L} \), let \( i_0 \) be the largest index with \( z_{i_0} \neq 0 \) and set \( \lambda := \{z_{i_0}\} \), where \( \{\cdot\} \) denotes the fractional part.
		\item If \( y \in \mathcal{L} \), set \( \lambda := 1 \).
	\end{enumerate}
	Then, 
	\[
	c_1 := \max_{1 \leq j \leq k} \left\{ \dfrac{\|b_1\|}{\|b_j^*\|} \right\} \quad \text{and} \quad \delta := \lambda \dfrac{\|b_1\|}{c_1}.
	\]
\end{lemma}

In applications, we encounter real numbers \( \eta_0, \eta_1, \ldots, \eta_k \), linearly independent over \( \mathbb{Q} \), and constants \( c_3, c_4 > 0 \) such that
\begin{align}\label{2.9m}
	|\eta_0 + x_1 \eta_1 + \cdots + x_k \eta_k| \leq c_3 \exp(-c_4 H),
\end{align}
where \( |x_i| \leq X_i \) for given bounds \( X_i \) (\( 1 \leq i \leq k \)). Let \( X_0 := \max_{1 \leq i \leq k} \{X_i\} \). Following \cite{Weg}, we approximate \eqref{2.9m} using a lattice generated by the columns of
$$ \mathcal{A}=\begin{pmatrix}
	1 & 0 &\ldots& 0 & 0 \\
	0 & 1 &\ldots& 0 & 0 \\
	\vdots & \vdots &\vdots& \vdots & \vdots \\
	0 & 0 &\ldots& 1 & 0 \\
	\lfloor C\eta_1\rfloor & \lfloor C\eta_2\rfloor&\ldots & \lfloor C\eta_{k-1}\rfloor& \lfloor C\eta_{k} \rfloor
\end{pmatrix} ,$$
where \( C \) is a large constant (typically \( \approx X_0^k \)). For an LLL-reduced basis \( b_1, \ldots, b_k \) of \( \mathcal{L} \) and \( y = (0, \ldots, -\lfloor C \eta_0 \rfloor) \), Lemma \ref{lem2.5m} yields a lower bound \( l(\mathcal{L}, y) \geq c_1 \). The following result is Lemma VI.1 in \cite{SMA}.

\begin{lemma}[Lemma VI.1 in \cite{SMA}]\label{lem2.6m}
	Let \( S := \sum_{i=1}^{k-1} X_i^2 \) and \( T := \dfrac{1 + \sum_{i=1}^{k} X_i}{2} \). If \( \delta^2 \geq T^2 + S \), then \eqref{2.9m} implies either \( x_1 = \cdots = x_{k-1} = 0 \) and \( x_k = -\dfrac{\lfloor C \eta_0 \rfloor}{\lfloor C \eta_k \rfloor} \), or
	\[
	H \leq \dfrac{1}{c_4} \left( \log(C c_3) - \log \left( \sqrt{\delta^2 - S} - T \right) \right).
	\]
\end{lemma}

All computations in this work were performed using Python.

\section{Proof of Theorem \ref{thm1.1l}}
\subsection{Bounds for \(n\) in relation to \(\ell\) and \(m\)}

From Eq.~\eqref{eq:main}, it follows that the palindromic representation of \(L_n^{(k)}\) consists of exactly \(2\ell + m\) digits. Consequently, any solution to \eqref{eq:main} must satisfy the inequalities
\[
10^{2\ell + m - 1} < L_n^{(k)} < 10^{2\ell + m}.
\]
By combining these bounds with \eqref{2.2}, we derive
\[
10^{2\ell + m - 1} < L_n^{(k)} \leq 2\alpha^{n} \quad \text{and} \quad \alpha^{n - 1} \leq L_n^{(k)} < 10^{2\ell + m}.
\]
Taking logarithms of these expressions yields the double inequality
\begin{equation}\label{4.1}
2\ell + m -3 < n < 5(2\ell + m) + 1,
\end{equation}
which holds for all integers \(n \geq 8\) and \(\ell, m \geq 1\).
\subsection{The case $n\leq k$}
Combining \eqref{eq:main} with \eqref{eq2.2}, we obtain the representation:
\begin{align}
	3\cdot 2^{n - 2} &= \underbrace{d_1\ldots d_1}_{\ell\ \text{digits}} \underbrace{d_2\ldots d_2}_{m\ \text{digits}} \underbrace{d_1\ldots d_1}_{\ell\ \text{digits}} \notag\\
	&= \underbrace{d_1\ldots d_1}_{\ell} \cdot 10^{\ell + m} + \underbrace{d_2\ldots d_2}_{m} \cdot 10^\ell + \underbrace{d_1\ldots d_1}_{\ell} \notag\\
	&= d_1 \left( \frac{10^\ell - 1}{9} \right) \cdot 10^{\ell + m} + d_2 \left( \frac{10^m - 1}{9} \right) \cdot 10^\ell + d_1 \left( \frac{10^\ell - 1}{9} \right).
	\label{eq:pal}
\end{align}

This expression can be simplified to:
\[
27 \cdot 2^{n - 2} = d_1 (10^\ell - 1) \cdot 10^{\ell + m} + d_2 (10^m - 1) \cdot 10^\ell + d_1 (10^\ell - 1),
\]
with the conditions \( d_1, d_2 \in \{0, 1, \ldots, 9\} \), \( d_1 > 0 \), \( d_1 \ne d_2 \), and \( \ell, m \geq 1 \).
Considering the equation modulo \( 10^\ell \) yields:
\[
27 \cdot 2^{n - 2} \equiv d_1 (10^\ell - 1) \pmod{10^\ell}.
\]
For \( n \geq 6 \), observe that both \( 27 \cdot 2^{n - 2} \) and \( 10^\ell \) are divisible by 16 when \( \ell \geq 4 \). However, \( d_1 (10^\ell - 1) \) cannot be divisible by 16 since \( 10^\ell - 1 \) is odd and \( d_1 < 10 \). This contradiction establishes that \( \ell \leq 3 \) for \( n \geq 6 \).

An alternative formulation of \eqref{eq:pal} gives:
\[
27 \cdot 2^{n - 2} = (d_1 \cdot 10^\ell + (d_2 - d_1)) \cdot 10^{\ell + m} + (d_1 - d_2) \cdot 10^\ell - d_1.
\]
The quantity \( |(d_1 - d_2) \cdot 10^\ell - d_1| \) is bounded above by \( 9 \cdot 10^3 + 9 < 2^{14}  \) in absolute value. Consequently, for \( n \geq 16 \), the term \( 27 \cdot 2^{n - 2} \) must be divisible by $2^{14}$. On the right-hand side above, the first term is divisible by $2^{14}$ when \( \ell + m \geq 14 \), but the final term $(d_1 - d_2) \cdot 10^\ell - d_1$ fails this divisibility condition for \( \ell \in \{1, 2, 3\} \). We therefore conclude that \( \ell + m \leq 13 \), implying \( m \leq 12 \) when \( n \geq 16 \).

To finalize our analysis, we conducted an exhaustive search using Python to examine all palindromic forms satisfying \eqref{eq:main} with \( \ell \leq 3 \) and \( m \leq 12 \). Our computations reveal that no such palindromic numbers exist as exactly having prime divisors 3 and 2. 

\subsection{The case \( n \geq k + 1 \)}
\subsubsection{A bound on \( n \) in terms of \( \ell \)}

We establish a series of auxiliary results to derive the desired estimates. First, we prove the following lemma.

\begin{lemma}\label{lem:l}
	Let $(\ell, m, n, k)$ be a solution to the Diophantine equation \eqref{eq:main} with $n \ge 8$, $k \ge 3$, and $n \ge k+1$. Then
	$$\ell < 2.62 \cdot 10^{12}k^4 (\log k)^2 \log n.$$
\end{lemma}

\begin{proof}
	Consider positive integers $n$ and $k$ satisfying \eqref{eq:main} with $n \ge k+1$ and $k \ge 3$. Rewriting \eqref{3.5} using \eqref{eq:pal}, we obtain
	\begin{align*}
		\left|\dfrac{1}{9}\left(d_1 \cdot 10^{2\ell+m} - (d_1-d_2) \cdot 10^{\ell+m} + (d_1-d_2) \cdot 10^{\ell} - d_1 \right) - f_k(\alpha)(2\alpha-1)\alpha^{n-1}\right| < \dfrac{3}{2}.
	\end{align*}
	Consequently, for all $\ell, m \ge 1$, we have
	\begin{align*}
		\left|f_k(\alpha)(2\alpha-1)\alpha^{n-1} - \dfrac{1}{9}\left(d_1 \cdot 10^{2\ell+m}\right)\right|
		&< \dfrac{3}{2} + \dfrac{1}{9}\left(|d_1-d_2| \cdot 10^{\ell+m} + |d_1-d_2| \cdot 10^{\ell} + d_1 \right) \\
		&\le \dfrac{3}{2} + 10^{\ell+m} + 10^{\ell} + 1 < 1.2 \cdot 10^{\ell+m}.
	\end{align*}
	Dividing both sides by $(d_1 \cdot 10^{2\ell+m})/9$ yields
	\begin{align}\label{gam1}
		\left|\Gamma_1\right| := \left|\dfrac{9f_k(\alpha)(2\alpha-1)}{d_1} \cdot 10^{-2\ell-m}\alpha^{n-1} - 1\right|
		&< 11 \cdot 10^{-\ell}.
	\end{align}
	Note that $\Gamma_1 \ne 0$, since otherwise we would obtain
	\[
	d_1 \cdot 10^{2\ell + m} = 9 f_k(\alpha)(2\alpha-1) \alpha^{n-1}.
	\]
	Applying a Galois automorphism from the decomposition field of $\Psi_k(x)$ over $\mathbb{Q}$ that maps $\alpha$ to a conjugate $\alpha_j$ for some $2 \leq j \leq k$, and taking absolute values, we arrive at the inequality
	\[
	10^3 \le d_1 \cdot 10^{2\ell + m} = \left| 9 f_k(\alpha_j)(2\alpha_j-1) \alpha_j^{n-1} \right| < 9,
	\]
	which is impossible.
	
	The relevant algebraic number field is $\mathbb{K} := \mathbb{Q}(\alpha)$ with degree $D = k$. Setting $t := 3$, we define
	\begin{equation}\nonumber
		\begin{aligned}
			\gamma_{1} &:= ({9f_k(\alpha)(2\alpha-1)})/d_1, \quad \gamma_{2} := \alpha, \qquad \gamma_{3} := 10, \\
			b_{1} &:= 1, \qquad \qquad b_{2} := n-1, \qquad b_{3} := -2\ell-m.
		\end{aligned}
	\end{equation}
	Using \eqref{eqh}, we estimate
	\[
	h(\gamma_1) = h\left({9f_k(\alpha)(2\alpha-1)})/d_1\right) < \log 9 + 3 \log k + \log 9 + \log 2 + \log 2 + \log 1 + (\log \alpha)/k < 9\log k + 0.7/k.
	\]
	Additionally, $h(\gamma_2) = h(\alpha) < (\log \alpha)/k < 0.7/k < 1/k$ and $h(\gamma_3) = h(10) = \log 10$. Thus, we take $A_1 := 9k\log k + 0.7$, $A_2 := 0.7$, and $A_3 := k \log 10$. Since \eqref{4.1} implies $2\ell + m < n$, we set $B := n$. Applying Theorem \ref{thm:Matl}, we obtain
	\begin{align}\label{gam2}
		\log |\Gamma_1| &> -1.4 \cdot 30^{6} \cdot 3^{4.5} \cdot k^2 (1 + \log k)(1 + \log n) \cdot (9k\log k + 0.7) \cdot 0.7 \cdot k \log 10 \nonumber \\
		&> -6.02 \cdot 10^{12}k^4 (\log k)^2 \log n.
	\end{align}
	for all $k \ge 3$ and $n \ge 8$. Comparing \eqref{gam1} and \eqref{gam2}, we derive
	\begin{align*}
		\ell \log 10 - \log 11 &< 	6.02 \cdot 10^{12}k^4 (\log k)^2 \log n,
	\end{align*}
	which implies $\ell < 2.62 \cdot 10^{12}k^4 (\log k)^2 \log n$.
\end{proof}

Next, we establish the following result.

\begin{lemma}\label{lem:m}
	Let $(\ell, m, n, k)$ be a solution to \eqref{eq:main} with $n \ge 8$, $k \ge 3$, and $n \ge k+1$. Then
	$$m < 1.73 \cdot 10^{24}k^8 (\log k)^3 (\log n)^2.$$
\end{lemma}

\begin{proof}
	Following the approach used in proving Lemma \ref{lem:l}, we rewrite \eqref{3.5} using \eqref{eq:pal} as
	\begin{align*}
		\left|\dfrac{1}{9}\left(d_1 \cdot 10^{2\ell+m} - (d_1-d_2) \cdot 10^{\ell+m} + (d_1-d_2) \cdot 10^{\ell} - d_1 \right) - f_k(\alpha)(2\alpha-1)\alpha^{n-1}\right| < \dfrac{3}{2}.
	\end{align*}
	For $\ell \ge 1$, this gives
	\begin{align*}
		\left|f_k(\alpha)(2\alpha-1)\alpha^{n-1} - \dfrac{1}{9}\left(d_1 \cdot 10^{2\ell+m} - (d_1-d_2) \cdot 10^{\ell+m}\right)\right|
		&< \dfrac{3}{2} + \dfrac{1}{9}(|d_1-d_2| \cdot 10^{\ell} + d_1) \\
		&\le \dfrac{3}{2} + 10^{\ell} + 1 < 1.3 \cdot 10^{\ell}.
	\end{align*}
	Dividing by $(d_1 \cdot 10^{2\ell+m} - (d_1-d_2) \cdot 10^{\ell+m})/9$ leads to
	\begin{align}\label{gam3}
		\left|\Gamma_2\right| := \left|\dfrac{9 f_k(\alpha)(2\alpha-1)}{(d_1 \cdot 10^{\ell} - (d_1-d_2))} \cdot 10^{-\ell-m}\alpha^{n-1} - 1\right|
		&< 12 \cdot 10^{-m}.
	\end{align}
	Again, $\Gamma_2 \ne 0$, as otherwise
	\[
	d_1 \cdot 10^{2\ell+m} - (d_1-d_2) \cdot 10^{\ell+m} = 9 f_k(\alpha)(2\alpha-1) \alpha^{n-1}.
	\]
	Applying a Galois automorphism mapping $\alpha$ to $\alpha_j$ for $2 \leq j \leq k$ and taking absolute values yields
	\[
	10^2 \le d_1 \cdot 10^{2\ell+m} - (d_1-d_2) \cdot 10^{\ell+m} = \left| 9 f_k(\alpha_j)(2\alpha_j-1) \alpha_j^{n-1} \right| < 9,
	\]
	a contradiction. The number field remains $\mathbb{K} := \mathbb{Q}(\alpha)$ with $D = k$ and $t := 3$,
	\begin{equation}\nonumber
		\begin{aligned}
			\gamma_{1} &:= 9f_k(\alpha)(2\alpha-1)/(d_1 \cdot 10^{\ell} - (d_1-d_2)), \quad \gamma_{2} := \alpha, \qquad \gamma_{3} := 10, \\
			b_{1} &:= 1, \qquad \qquad \qquad \qquad \qquad \qquad b_{2} := n-1, \qquad b_{3} := -\ell-m.
		\end{aligned}
	\end{equation}
	We estimate
	\begin{align*}
		h(\gamma_{1}) &= h(9f_k(\alpha)(2\alpha-1)/(d_1 \cdot 10^{\ell} - (d_1-d_2))) \\
		&\le h(9) + h(f_k(\alpha)) + h(2\alpha-1) + h(d_1 \cdot 10^{\ell} - (d_1-d_2)) \\
		&< \log 9 + 2\log k + \log 2 + \log(\alpha)/k + \log 2 + \log 9 + \ell \log 10 + \log 9 + \log 2 \\
		&< 3\log 9 + 3\log 2 + (2.62 \cdot 10^{12}k^4 (\log k)^2 \log n)\log 10 + 2\log k + (\log \alpha)/k \\
		&= k^4 (\log k)^2 \log n \left(\frac{3\log 9}{k^4(\log k)^2 \log n} + \frac{3\log 2}{k^4(\log k)^2 \log n} + \frac{\log \alpha}{k^5(\log k)^2 \log n} + 2.62 \cdot 10^{12}\log 10\right) \\
		&< 6.04 \cdot 10^{12}k^4 (\log k)^2 \log n,
	\end{align*}
	so we take $A_1 := 6.04 \cdot 10^{12}k^5 (\log k)^2 \log n$. As before, $A_2 := 0.7$, $A_3 := k \log 10$, and $B := n$. 
	
	By Theorem \ref{thm:Matl},
	\begin{align}\label{gam4}
		\log |\Gamma_2| &> -1.4 \cdot 30^{6} \cdot 3^{4.5} \cdot k^2 (1 + \log k)(1 + \log n) \cdot 6.04 \cdot 10^{12}k^5 (\log k)^2 \log n \cdot 0.7 \cdot k \log 10 \nonumber \\
		&> -3.95 \cdot 10^{24}k^8 (\log k)^3 (\log n)^2.
	\end{align}
	for $k \ge 3$ and $n \ge 8$. Comparing \eqref{gam3} and \eqref{gam4}, we obtain
	\begin{align*}
		m \log 10 - \log 12 &< 3.95 \cdot 10^{24}k^8 (\log k)^3 (\log n)^2,
	\end{align*}
	which implies $m < 1.73 \cdot 10^{24}k^8 (\log k)^3 (\log n)^2$.
\end{proof}

Finally, returning to \eqref{4.1} and using Lemmas \ref{lem:l} and \ref{lem:m}, we have
\begin{align*}
	n &< 5(2\ell + m) + 1 \\
	&< 5(2 \cdot 2.62 \cdot 10^{12}k^4 (\log k)^2 \log n + 1.73 \cdot 10^{24}k^8 (\log k)^3 (\log n)^2) + 1 \\
	&= k^8 (\log k)^3 (\log n)^2 \left(\frac{5 \cdot 2 \cdot 2.62 \cdot 10^{12}k^4 (\log k)^2 \log n}{k^8 (\log k)^3 (\log n)^2} + 5 \cdot 1.73 \cdot 10^{24} + \frac{1}{k^8 (\log k)^3 (\log n)^2} \right)  \\
	&< 8.66 \cdot 10^{24}k^8 (\log k)^3 (\log n)^2,
\end{align*}
yielding
\begin{align}\label{boundn}
	n &< 1.63 \cdot 10^{29}k^8 (\log k)^5.
\end{align}

We now consider two cases based on the value of $k$.

\newpage
\textbf{Case I:} When $k \leq 1500$, it follows from \eqref{boundn} that  
\[
n < 1.63 \cdot 10^{29}k^8 (\log k)^5.
\]
We proceed to reduce this upper bound for $n$. First, we recall \eqref{gam1} as,
\[
\Gamma_1 := \frac{9}{d_1} \cdot 10^{-2\ell - m} f_k(\alpha)(2\alpha-1) \alpha^{n-1} - 1 = e^{\Lambda_1} - 1.
\]
We already showed that $\Gamma_1 \neq 0$, thus $\Lambda_1 \neq 0$. If $\ell \geq 2$, we have $\left| e^{\Lambda_1} - 1 \right| = |\Gamma_1| < 0.5$, which yields $e^{|\Lambda_1|} \leq 1 + |\Gamma_1| < 1.5$. Thus, 
\[
\left| (n-1) \log \alpha - (2\ell + m) \log 10 + \log \left( \frac{9 f_k(\alpha)(2\alpha-1)}{d_1} \right) \right| < \frac{18}{10^\ell}.
\]
We proceed by applying the LLL algorithm for each value of \( k \in [3,1500] \) and for every \( d_1 \in \{1, \dots, 9\} \), in order to determine a lower bound for the smallest nonzero value of the linear form above, in which all our integer coefficients are bounded above in absolute value by  \( n < 8.8 \cdot 10^{58} \). For this purpose, we consider the lattice
 
\[
\mathcal{A} = \begin{pmatrix} 
	1 & 0 & 0 \\ 
	0 & 1 & 0 \\ 
	\lfloor C\log \alpha\rfloor & \lfloor C\log (1/10)\rfloor & \lfloor C\log \left(9 f_k(\alpha)(2\alpha-1)/d_1\right) \rfloor
\end{pmatrix},
\]
with $C := 2.1\cdot 10^{178}$ and $y := (0,0,0)$. Applying Lemma \ref{lem2.5m}, we obtain
\[
l(\mathcal{L},y) = |\Lambda| > c_1 = 10^{-60} \quad \text{and} \quad \delta = 1.81\cdot 10^{59}.
\]
Via Lemma \ref{lem2.6m}, we have that $S =1.53 \cdot 10^{118}$ and $T = 1.32 \cdot 10^{59}$. Since $\delta^2 \geq T^2 + S$, choosing $c_3 := 18$ and $c_4 := \log 10$, we get $\ell \leq 121$.

Next, we go back to \eqref{gam3}, that is,
$$\Gamma_2:=\dfrac{9 f_k(\alpha)(2\alpha-1)}{(d_1 \cdot 10^{\ell} - (d_1-d_2))} \cdot 10^{-\ell-m}\alpha^{n-1} - 1=e^{\Lambda_2}-1,$$
for which we have already established that \( \Gamma_2 \ne 0 \), and hence \( \Lambda_2 \ne 0 \). Let us now assume, temporarily, that \( m \ge 2 \). As in the previous arguments, this assumption implies that
\begin{align*}
	\left|(n-1) \log \alpha - (\ell + m) \log 10 + \log \left( \dfrac{9 f_k(\alpha)(2\alpha-1)}{d_1\cdot 10^{\ell}-(d_1-d_2)} \right)\right|<\dfrac{19}{10^m}.
\end{align*}
For every \( k \in [3, 1500] \), \( \ell \in [1,121] \), and for all \( d_1, d_2 \in \{0,\ldots,9\} \) such that \( d_1 > 0 \) and \( d_1 \ne d_2 \), we apply the LLL algorithm to estimate a lower bound for the smallest nonzero value of the linear form, where the integer coefficients are bounded in absolute value by \( n < 8.8 \cdot 10^{58} \). To this end, we employ the approximation lattice
$$ \mathcal{A}=\begin{pmatrix}
	1 & 0 & 0 \\
	0 & 1 & 0 \\
	\lfloor C\log \alpha\rfloor & \lfloor C\log (1/10)\rfloor& \lfloor C\log \left(9 f_k(\alpha)(2\alpha-1)/(d_1\cdot 10^{\ell}-(d_1-d_2))\right)\rfloor
\end{pmatrix} ,$$
Let \( C := 3.0 \cdot 10^{178} \) and \( y := (0, 0, 0) \). By applying Lemma~\ref{lem2.5m}, we obtain
\[
\ell(\mathcal{L}, y) = |\Lambda| > c_1 = 10^{-60}, \quad \text{with} \quad \delta = 2.0 \cdot 10^{59}.
\]
Nevertheless, Lemma~\ref{lem2.6m} yields the same values for \( S \) and \( T \) as before. Therefore, choosing \( c_3 := 19 \) and \( c_4 := \log 10 \), we conclude that \(m \leq 122.
\)

Consequently, by applying inequality~\eqref{4.1}, we obtain the upper bound \( n \leq 1821 \). To conclude this case, we carried out a computational search using Python, checking all values of \( L_n^{(k)} \) for \( k \in [3, 1500] \) and \( n \in [7, 1821] \), in order to identify those that are palindromic concatenations of two distinct repdigits. However, this exhaustive check returned no results, and thus no such values were found.

\textbf{Case II:} If $k > 1500$, then 
\begin{align}\label{boundn2}
	n	<1.63\cdot 10^{29}k^8 (\log k)^5<2^{k/2}.
\end{align}
We prove the following result.
\begin{lemma}\label{lem:nm}
	Let $(\ell, m, n, k)$ be a solution to the Diophantine equation \eqref{eq:main} with $n \ge 8$, $k>1500$ and $n\ge k+1$. Then
	$$k<1.8\cdot 10^{31} \qquad\text{and}\qquad n<3.5\cdot 10^{288}.$$
\end{lemma}
\begin{proof}
Since $k>1500$, then \eqref{boundn2} holds and we can use the sharper estimate in \eqref{pk_b1} together with \eqref{eq:pal} and write
\begin{align*}
	\left| 
	\dfrac{1}{9} \left( 
	d_1 \cdot 10^{2\ell + m} 
	- (d_1 - d_2) \cdot 10^{\ell + m} 
	+ (d_1 - d_2) \cdot 10^{\ell} 
	- d_1 
	\right) 
	- 3 \cdot 2^{n - 2} 
	\right|
	&< 3 \cdot 2^{n - 2} \cdot \frac{36}{2^{k/2}}.\end{align*}
Thus
\begin{align*}
	\left| \dfrac{1}{9} \cdot d_1 \cdot 10^{2\ell + m} - 3 \cdot 2^{n - 2} \right|
	&< 3 \cdot 2^{n - 2} \cdot \frac{36}{2^{k/2}} 
	+ \frac{1}{9} \left( 
	|d_1 - d_2| \cdot 10^{\ell + m} 
	+ |d_1 - d_2| \cdot 10^{\ell} 
	+ d_1 
	\right)\\&< 3 \cdot 2^{n - 2} \cdot \frac{36}{2^{k/2}} +1.3\cdot10^{\ell+m},
\end{align*}
Dividing both sides by $2^{n - 2}$ gives
\begin{align*}
	\left|\dfrac{1}{27}\cdot d_1\cdot 10^{2\ell+m}\cdot 2^{-(n-2)}-1\right|<\frac{36}{2^{k/2}} +1.3\cdot\frac{10^{\ell+m}}{2^{n-2}}.
\end{align*}
Using the result in \eqref{4.1}, we get
\begin{align*}
	\left|\dfrac{1}{27}\cdot d_1\cdot 10^{2\ell+m}\cdot 2^{-(n-2)}-1\right|&<\frac{36}{2^{k/2}} +1.3\cdot\frac{10^{\ell+m}}{2^{n-2}}=\frac{36}{2^{k/2}} +1.3\cdot\frac{10^{\ell+m}}{2^{n-1}\cdot 2^3}\\&<\frac{36}{2^{k/2}} +1.3\cdot\frac{10^{\ell+m}}{8\cdot10^{2\ell+m-1}}=\frac{36}{2^{k/2}} +1.3\cdot\frac{10}{8\cdot2^{l}}\\&<\frac{36}{2^{k/2}} +\frac{2}{10^{l}}<\frac{39}{2^{\min\left\{k/2,\, \ell \log_2 10\right\}}}.
\end{align*}
Therefore,
\begin{align}\label{4.9}
	\left|\dfrac{1}{27}\cdot d_1\cdot 10^{2\ell+m}\cdot 2^{-(n-2)}-1\right|&<\frac{39}{2^{\min\left\{k/2,\, \ell \log_2 10\right\}}}.
\end{align}
Fix
\begin{align*}
	\Gamma_3:=\dfrac{1}{27}\cdot d_1\cdot 10^{2\ell+m}\cdot 2^{-(n-2)}-1.
\end{align*}
Suppose $\Gamma_3= 0$, then
\[d_1\cdot 10^{2\ell+m}=27\cdot 2^{(n-2)}.
\]
This is a contradiction since the left hand side is divisible by 5, however, the right hand side is not. Clearly, $\Gamma_3\neq0$. We take the algebraic number field $\mathbb{K}=\mathbb{Q}$ with $D=1, t:=3$,
\begin{equation}\nonumber
	\begin{aligned}
		&\gamma_{1}:=d_1/27,\quad~&\gamma_{2}:=10,\qquad\qquad&\gamma_{3}:=2,\\
		&b_{1}:=1,\qquad\qquad &b_{2}:=2\ell+m,\qquad &b_{3}:=-(n-2).
	\end{aligned}
\end{equation}
Since $h(\gamma_{1})\leq h(d_1)+h(27)\leq\log(243), h(\gamma_{2})=h(10)=\log10 $ and $h(\gamma_{3})=h(2)=\log2$, we get $A_1:=\log(243), A_2:=\log(10), A_3:=\log2.$ 
Also, $B :=n>\max\{|b_i|:i=1,2,3\}$. Applying Theorem \ref{thm:Matl}, we get
\begin{align}\label{gam6}
	\log |\Gamma_3| &> -1.4\cdot 30^{6} \cdot 3^{4.5}\cdot 1^2 (1+\log 1)(1+\log n)\cdot \log 243\cdot \log 2 \cdot\log 10 \nonumber\\
	&> -1.9\cdot 10^{12}\log n.
\end{align}
for all $n\ge 8$. Comparing \eqref{4.9} and \eqref{gam6}, we get
\begin{align*}
	\min\{k/2,\, \ell\log_2 10\}\log 2-\log 39 &<1.9\cdot 10^{12}\log n.
\end{align*}
This yields two cases 	.
\begin{enumerate}[(a)]
	\item If $\min\{k/2,\, \ell\log_2 10\}:=k/2$, then $(k/2)\log 2-\log 39 <1.9\cdot 10^{12}\log n$. Hence,
	\begin{align*}
		k &<5.5\cdot 10^{12}\log n.
	\end{align*}
By \eqref{boundn2}, we have \[n<1.63\cdot 10^{29}k^8(\log k)^5,\] and so \[\log n < 68+13\log k <23\log k,\] for $k>1500$. Hence we get \[k<8.3\cdot 10^{15}.\]
	
	\item If $\min\{k/2,\, \ell\log_2 10\}:=\ell\log_2 10$, then $(\ell\log_2 10)\log 2-\log 39 <1.9\cdot 10^{12}\log n$, which yields
	\begin{align*}
		\ell &<8.3\cdot 10^{11}\log n.
	\end{align*}
We proceed as in \eqref{4.9} via
\begin{align*}
	\left| 
	\dfrac{1}{9} \left( 
	d_1 \cdot 10^{2\ell + m} 
	- (d_1 - d_2) \cdot 10^{\ell + m} 
	+ (d_1 - d_2) \cdot 10^{\ell} 
	- d_1 
	\right) 
	- 3 \cdot 2^{n - 2} 
	\right|
	&< 3 \cdot 2^{n - 2} \cdot \frac{36}{2^{k/2}}.\end{align*}
Therefore
\begin{align*}
	\left| 
	\dfrac{1}{9} \left( 
	d_1 \cdot 10^{2\ell + m} 
	- (d_1 - d_2) \cdot 10^{\ell + m} 
	\right) - 3 \cdot 2^{n - 2} 
	\right|
	&< 3 \cdot 2^{n - 2} \cdot \frac{36}{2^{k/2}} + \frac{1}{9}(|d_1 - d_2|\cdot 10^{\ell}+ d_1)\\&<3 \cdot 2^{n - 2} \cdot \frac{36}{2^{k/2}} +1.2\cdot 10^l.
	\end{align*}
Dividing both sides by $3 \cdot 2^{n - 2}$, we get that
\begin{align*}
	\left|\dfrac{1}{27}(d_1\cdot 10^{2\ell+m}-(d_1-d_2)\cdot10^{\ell+m})\cdot 2^{-(n-2)}-1\right| &
	< \frac{36}{2^{k/2}} +\frac{10^l}{2^{n-2}}.
\end{align*}
Notice that since $n\geq k+1, n-2>k/2 $ and $ k/2 > l\log 10,  $ then we get 
\begin{align}\label{gam7}
	\left|\dfrac{1}{27}(d_1\cdot 10^{2\ell+m}-(d_1-d_2)\cdot10^{\ell+m})\cdot 2^{-(n-2)}-1\right| &
	< \frac{7}{2^{k/2}}.
\end{align}
Fix \[\Gamma_4:=\dfrac{1}{27}(d_1\cdot 10^{\ell}-(d_1-d_2))\cdot10^{\ell+m}\cdot 2^{-(n-2)}-1.\]
Clearly $\Gamma_4\ne 0$, otherwise we would have
\[d_1\cdot 10^{2\ell+m}-(d_1-d_2)\cdot10^{\ell+m}=27\cdot 2^{(n-2)}
,\]
which is a contradiction because the left hand side is divisible by 5 while the right hand side is not. Fix $\mathbb{K} := \mathbb{Q}.$ We have $D = 1$, $t :=3$,
\begin{equation}\nonumber
	\begin{aligned}
		&\gamma_{1}:=(d_1\cdot 10^{\ell}-(d_1-d_2))/27,\quad~&\gamma_{2}:=10,\qquad\qquad&\gamma_{3}:=2,\\
		&b_{1}:=1,\qquad \qquad\qquad\qquad\qquad\qquad\qquad~ &b_{2}:=\ell+m,\qquad~~ &b_{3}:=-(n-2).
	\end{aligned}
\end{equation}
\begin{align*}
	h(\gamma_{1})&=h((d_1\cdot 10^{\ell}-(d_1-d_2))/27)\\&<\log 9 +\ell \log 10 +2\log 9 +\log 27 +4\log 2\\&< \log 9 +(8.3\cdot 10^{11}\log n) \log 10 +2\log 9 +\log 27 +4\log 2\\&<1.9\cdot 10^{12} \log n.
\end{align*}Taking $A_1:=1.9\cdot 10^{12} \log n, A_2:=\log 2, A_3:=\log 10$ and $B:=n$. 
By Theorem \ref{thm:Matl}, we have
\begin{align}\label{gam8}
	\log |\Gamma_4| &> -1.4\cdot 30^{6} \cdot 3^{4.5}\cdot 1^2 (1+\log 1)(1+\log n)\cdot 1.9\cdot 10^{12}\log n\cdot \log 2 \cdot \log 10\nonumber\\
	&> -6.5\cdot 10^{23}(\log n)^2,
\end{align}
for all $k\geq3$ and $n\ge 7$. Comparing \eqref{gam7} and \eqref{gam8}, we get
\begin{align*}
	(k/2)\log 2-\log 7 &<6.5\cdot 10^{23}(\log n)^2,
\end{align*}
hence \[k<1.9\cdot 10^{24}(\log n)^2.\]
From \eqref{boundn2}, $\log n <24\log k$ for $k>1500$. This gives, \[k<1.8\cdot 10^{31},\] and \[n<3.5\cdot 10^{288}.\]	
\end{enumerate}
This completes the proof.
\end{proof}

We now proceed to reduce the bounds obtained in Lemma \ref{lem:nm}. To do this, we revisit \eqref{4.9} and recall that
\begin{align*}
	\Gamma_3:=\dfrac{1}{27}\cdot d_1\cdot 10^{2\ell+m}\cdot 2^{-(n-2)}-1
\end{align*}
We already showed that $\Gamma_3 \neq 0$, therefore $\Lambda_3 \neq 0$. Also since $\min \{k/2, \ell\log_2 10\}>6, $
\[
\left| e^{\Lambda_3} - 1 \right| = |\Gamma_3| < 0.5,
\]
which leads to $e^{|\Lambda_3|} \leq 1 + |\Gamma_3| < 1.5$. Therefore 
\[
\left|  \log \left(\dfrac{d_1}{27}\right) + (2\ell + m) \log 10 -(n-1) \log 2  \right| < \dfrac{59}{2^{\min\{k/2,\, \ell\log_2 10\}}}.
\]
For each $d_1\in\{1,\ldots , 9\}$, we apply the LLL-algorithm to get a lower bound for the smallest nonzero value of the above linear form, bounded by integer coefficients with absolute values less than $n<3.5\cdot 10^{288}$. Here, we consider the lattice  
\[
\mathcal{A}_1 = \begin{pmatrix} 
	1 & 0 & 0 \\ 
	0 & 1 & 0 \\ 
	\lfloor C\log (d_1/27)\rfloor & \lfloor C\log 10\rfloor & \lfloor C\log (1/2) \rfloor
\end{pmatrix},
\]
where we set $C := 1.3\cdot 10^{867}$ and $y := (0,0,0)$. Applying Lemma \ref{lem2.5m}, we obtain  
\[
l(\mathcal{L},y) = |\Lambda| > c_1 = 10^{-291} \quad \text{and} \quad \delta = 5.6\cdot 10^{290}.
\]
Using Lemma \ref{lem2.6m}, we conclude that $S = 2.5 \cdot 10^{578}$ and $T = 5.3 \cdot 10^{289}$. Since $\delta^2 \geq T^2 + S$, then choosing $c_3 := 59$ and $c_4 := \log 2$, we get  $\min\{k/2,\, \ell\log_2 10\} \leq 1921$. 

We proceed the analysis in two ways.
\begin{enumerate}[(a)]
	\item If $\min\{k/2,\, \ell\log_2 10\}:=k/2$, then $k/2 \leq 1921$, or $k\le 3842$.

	\item If $\min\{k/2,\, \ell\log_2 10\}:=\ell\log_2 10$, then $\ell\log_2 10\le 1921$. Therefore $\ell \le 579$.
	We revisit equation \eqref{gam7} and recall the expression
	\[\Gamma_4:=\dfrac{1}{27}(d_1\cdot 10^{\ell}-(d_1-d_2))\cdot10^{\ell+m}\cdot 2^{-(n-2)}-1.\]
	Since we showed that \(\Gamma_4 \neq 0\), it follows that \(\Lambda_4 \neq 0\). Also since $k>1500,$ we get the inequality \[|e^{\Lambda_4}-1|=|\Gamma_4|<0.5,\] which implies that $e^{|\Lambda_4|}\leq1+|\Gamma_4|<1.5$ Hence we arrive at
	\[
	\left|  \log \left(\frac{d_1\cdot 10^{\ell}-(d_1-d_2)}{27}\right) + (\ell + m) \log 10 -(n-2) \log 2  \right| < \frac{11}{2^{k/2}}.
	\]
We apply LLL-algorithm with a restriction on integer coefficients such that their absolute values are below $3.5\cdot 10^{288}.$ For this we consider the lattice
	\[
	\mathcal{A}_2 = \begin{pmatrix} 
		1 & 0 & 0 \\ 
		0 & 1 & 0 \\ 
		\lfloor C\log \left(\left(d_1\cdot 10^{\ell}-(d_1-d_2)\right)/27\right)\rfloor & \lfloor C\log 10\rfloor & \lfloor C\log (1/2) \rfloor
	\end{pmatrix},
	\]
	with \(C := 1.3^{867}\) and set \(y := (0,0,0)\) as before. Taking the same values for $\delta$, $S$ and $T$ as before, we choose \(c_3 := 11\) and \(c_4 := \log 2\). This yields $k/2 \leq 1919$ which gives $k\le 3838$.
	
\end{enumerate}
Therefore, we always have $k\le 3838$. 

Applying the bound obtained on $k$ to inequality \eqref{boundn2}, we get $n < 3\cdot 10^{62}$. We therefore proceed with a second round of reduction using this updated bound on $n$. Going back to equation \eqref{4.9}, as previously done, we use the LLL-algorithm to obtain a lower bound for the smallest nonzero value of the corresponding linear form, with the condition that the involved integer coefficients are bounded above by $n < 3.0\cdot 10^{62}$. We use the same lattice  $\mathcal{A}_1$ as before but now take $C := 9.0\cdot 10^{188}$ and $y := (0,0,0)$. By applying Lemma \ref{lem2.5m}, we get  
\[
l(\mathcal{L},y) = |\Lambda| > c_1 = 10^{-65} \quad \text{and} \quad \delta = 6.0\cdot 10^{64}.
\]
Using Lemma \ref{lem2.6m}, we conclude that $S = 1.9 \cdot 10^{126}$ and $T = 5.0 \cdot 10^{63}$. Since $\delta^2 \geq T^2 + S$, and choosing $c_3 := 59$ and $c_4 := \log 2$, it follows that $\min\{k/2,\, \ell\log_2 10\} \leq 419$. 

We again analyze two cases:
\begin{enumerate}[(a)]
	\item If $\min\{k/2,\, \ell\log_2 10\} := k/2$, then $k/2 \leq 419$, which implies $k \leq 838$. This contradicts  our assumption that $k>1500$.
	
	\item If $\min\{k/2,\, \ell\log_2 10\} := \ell\log_2 10$, then $\ell\log_2 10 \leq 419$, giving $\ell \leq 127$.
	
	Revisiting equation \eqref{gam7}, we again apply the LLL-algorithm, with integer coefficients restricted to absolute values below $n <  3.0\cdot10^{62}$. By the same lattice $\mathcal{A}_2$ as before, we have $C := 9.0\cdot10^{188}$ and $y := (0,0,0)$. Using Lemma \ref{lem2.5m} and Lemma \ref{lem2.6m}, we have $\delta=6.0\cdot 10^{64}$, $S = 1.9 \cdot 10^{126}$ and $T = 5.0\cdot 10^{63}$. Selecting $c_3 := 11$ and $c_4 := \log 2$, we deduce that $k/2 \leq 417$, leading to $k \leq 834$. This again contradicts  our original assumption that $k>1500$.
	
\end{enumerate}
Therefore, the proof is complete. \qed

\section*{Acknowledgments} 
The first author was funded by the cost centre 0730 of the Mathematics Division at Stellenbosch University.

\section*{Addresses}
$ ^{1} $ Mathematics Division, Stellenbosch University, Stellenbosch, South Africa.

Email: \url{hbatte91@gmail.com}\\
$ ^{2} $ Department of Mathematics, Makerere University, Kampala, Uganda.

Email: \url{kaggwaprosper58@gmail.com}

\newpage
\appendix
\section{Appendices}

\subsection{Py Code I}\label{app1}
\begin{Verbatim}
def k_lucas_numbers(k, n_max=2530):
	if k < 2:
		raise ValueError("k must be at least 2")
	L = [0] * (k - 2) + [2, 1]
	for n in range(len(L), n_max):
		L.append(sum(L[-k:]))
	return L
	
def is_palindromic_concatenation(num):
	s = str(num)
	n = len(s)
	for l in range(1, n // 2 + 1):
		m = n - 2 * l
		if m <= 0:
			continue
		part1, part2, part3 = s[:l], s[l:l + m], s[l + m:]
		if part1 == part3 and len(set(part1)) == 1 
		and len(set(part2)) == 1 and part1 != part2:
			return True
	return False
	
def find_all_palindromic_k_lucas(k_min=3, k_max=1500, n_max=1500):
	results = []
	for k in range(k_min, k_max + 1):
		lucas_seq = k_lucas_numbers(k, n_max + k - 3)
		for i, val in enumerate(lucas_seq):
			math_n = i + 3 - k
			if math_n < 8 or math_n > 1500:
				continue
			if val > 0 and is_palindromic_concatenation(val):
				results.append((k, math_n, val))
	return results
	
# Run locally
results = find_all_palindromic_k_lucas()
for k, n, v in results:
	print(f"L_{n}^({k}) = {v}")
\end{Verbatim}


\begin{thebibliography}{99}	
	\bibitem{banks}
	Banks, W. D., \& Luca, F.:
	Concatenations with binary recurrent sequences.
	\textit{J. Integer Seq.} \textbf{8}(5), Art. 05.1.3 (2005).
	
	\bibitem{BW}
	Baker, A., \& Wüstholz, G.:
	Logarithmic forms and group varieties, 19--62 (1993).\\
	\url{https://doi.org/10.1515/crll.1993.442.19}
	
	\bibitem{Lucas}
	Batte, H.:
	Lucas numbers that are palindromic concatenations of two distinct repdigits.
	\textit{Mathematica Pannonica} (2025).
	\url{https://doi.org/10.1556/314.2025.00003}
	
	\bibitem{guma}
	Batte, H., \& Guma, D.:
	On $k$-Pell numbers that are palindromes formed by two distinct repdigits (2025).
	\url{https://doi.org/10.48550/arXiv.2504.14261}
	
		
	\bibitem{kaggwa}
	Batte, H., \& Kaggwa, P.:
	Perrin numbers that are palindromic concatenations of two repdigits.
	\textit{Int. J. Math. Comput. Tech.} \textbf{8}(4), 1--12 (2024).
	\url{https://doi.org/10.5281/zenodo.13886753}
	
	\bibitem{Hbatte}
	Batte, H., \& Luca, F.:
	On the largest prime factor of the $k$-generalized Lucas numbers.
	\textit{Bol. Soc. Mat. Mex.} \textbf{30}(2), 33 (2024).
	\url{https://doi.org/10.1007/s40590-024-00604-9}
	
	\bibitem{Batte}
	Batte, H., \& Luca, F.:
	$k$-Fibonacci numbers that are palindromic concatenations of two distinct repdigits (2025).
	\url{https://doi.org/10.48550/arXiv.2504.10138}
		
	\bibitem{bravo2014}
	Bravo, J. J., \& Luca, F.:
	Repdigits in $k$--Lucas sequences.
	\textit{Proc. Math. Sci.} \textbf{124}, 141--154 (2014).
	\url{https://doi.org/10.1007/s12044-014-0174-7}
	
	\bibitem{Matveev}
	Bugeaud, Y., Mignotte, M., \& Siksek, S.:
	Classical and modular approaches to exponential Diophantine equations I. Fibonacci and Lucas perfect powers.
	\textit{Ann. Math.} \textbf{163}(3), 969--1018 (2006).
	\url{http://www.jstor.org/stable/20159981}
	
	\bibitem{Cas}
	Cassels, J. W. S.:
	\textit{An Introduction to the Geometry of Numbers}.
	Springer Science \& Business Media (2012).
	
	\bibitem{Padovan}
	Chalebgwa, T. P., \& Ddamulira, M.:
	Padovan numbers which are palindromic concatenations of two distinct repdigits.
	\textit{Rev. R. Acad. Cienc. Exactas Fís. Nat. Ser. A Math.} \textbf{115}(3), 108 (2021).
	\url{https://doi.org/10.1007/s13398-021-01047-x}
	
	\bibitem{Narayan}
	Ddamulira, M., Emong, P., \& Mirumbe, G. I.:
	Palindromic concatenations of two distinct repdigits in Narayana's cows sequence.
	\textit{Bull. Iran. Math. Soc.} \textbf{50}(3), 35 (2024).\\
	\url{https://doi.org/10.1007/s41980-024-00877-w}
	
	\bibitem{Weg}
	de Weger, B. M. M.:
	Solving exponential Diophantine equations using lattice basis reduction algorithms.
	\textit{J. Number Theory} \textbf{26}(3), 325--367 (1987).\\
	\url{https://doi.org/10.1016/0022-314X(87)90088-6}
	
	\bibitem{LLL}
	Lenstra, A. K., Lenstra, H. W., \& Lovász, L.:
	Factoring polynomials with rational coefficients.
	\textit{Math. Ann.} \textbf{261}, 515--534 (1982).
	
	\bibitem{miles}
	Miles, E. P., Jr.:
	Generalized Fibonacci numbers and associated matrices.
	\textit{Am. Math. Mon.} \textbf{67}, 745--752 (1960).
	\url{https://doi.org/10.1080/00029890.1960.11989593}
	
	\bibitem{gomez}
	Ruiz, C. A. G., \& Luca, F.:
	Multiplicative independence in $k$-generalized Fibonacci sequences.
	\textit{Lithuanian Math. J.} \textbf{56}, 503--517 (2016).
	\url{https://doi.org/10.1007/s10986-016-9332-1}
	
	\bibitem{SMA}
	Smart, N. P.:
	\textit{The Algorithmic Resolution of Diophantine Equations: A Computational Cookbook}, Vol. 41.
	Cambridge University Press (1998).
	
	\bibitem{Repdigit2}
	Qu, Y., \& Zeng, J.:
	Lucas numbers which are concatenations of two repdigits.
	\textit{Mathematics} \textbf{8}(8), 1360 (2020).
	\url{https://doi.org/10.3390/math8081360}
	
	\bibitem{wolfram}
	Wolfram, D. A.:
	Solving generalized Fibonacci recurrences.
	\textit{Fibonacci Quart.} \textbf{36}(2), 129--145 (1998).
	\url{https://doi.org/10.1080/00150517.1998.12428948}
	
	
\end{thebibliography}
\end{document}